\documentclass[11pt,a4paper]{amsart}
\usepackage[american]{babel}

\usepackage{graphicx,url,etoolbox}

\makeatletter
\patchcmd{\@settitle}{center}{flushleft}{}{}
\patchcmd{\@settitle}{center}{flushleft}{}{}
\patchcmd{\@setauthors}{\centering}{\raggedright}{}{}
\patchcmd{\abstract}{3pc}{0pt}{}{} % remove indentation
\makeatother
\usepackage{mathtools}
\usepackage{amsmath,amssymb,amsfonts,color}
\usepackage{algorithm}
\usepackage{algorithmic}
\usepackage{color}
\usepackage{booktabs}
\usepackage[mathscr]{eucal}
\usepackage{url}
\usepackage{dsfont}
\usepackage{subfloat}
\usepackage{subcaption}
\usepackage{tabularx}
\usepackage{hyperref}
\usepackage{caption}
\usepackage{subcaption,graphicx}
\usepackage{bm}
\usepackage{epstopdf}
\usepackage{xargs}
\usepackage[pdftex,dvipsnames]{xcolor}
\calclayout

 \numberwithin{equation}{section}

\newtheorem{lemma}{Lemma}

\newtheorem{proposition}{Proposition}

\newtheorem{remark}{Remark}
\newtheorem{definition}{Definition}
\newtheorem{assumption}{Assumption}

\newcommand{\cA}{\mathcal{A}}

\newcommand{\R}{\mathbb{R}}
\newcommand{\N}{\mathbb{N}}

\newcommand{\cP}{\mathcal{P}}

\newcommand{\cU}{\mathcal{U}}

\newcommand{\cJ}{\mathcal{J}}
\newcommand{\cH}{\mathcal{H}}
\newcommand{\cL}{\mathcal{L}}

\newcommand{\E}{\mathbb{E}}
\newcommand{\px}[1]{\partial_{{x} _{#1}}}

\newcommand{\norm}[1]{\left\lVert #1\right\rVert}

\def\one{\mbox{1\hspace{-4.25pt}\fontsize{12}{14.4}\selectfont\textrm{1}}}

\definecolor{philipp}{RGB}{205, 102, 000}

\usepackage[colorinlistoftodos,prependcaption,textsize=tiny]{todonotes}

\begin{document}

%\headers{Policy iteration for HJB equations with control constraints}{S. Kundu and K. Kunisch}
\title[Policy iteration for HJB equations with control constraints]{Policy iteration for Hamilton-Jacobi-Bellman equations with control constraints}	

\author{Sudeep Kundu\textsuperscript{$\dagger$}}
\thanks{\textsuperscript{$\dagger$}Institute for Mathematics and Scientific Computing, University of Graz, Heinrichstrasse 36, A-8010 Graz, Austria, ({\tt
		sudeep.kundu@uni-graz.at}).}
\author {Karl Kunisch\textsuperscript{$*$}}
\thanks{\textsuperscript{$*$}Institute for Mathematics and Scientific Computing, University of Graz, Heinrichstrasse 36, A-8010 Graz, Austria and
		Radon Institute for Computational and Applied Mathematics (RICAM), Altenbergerstra{\ss}e 69, A-4040 Linz, Austria, ({\tt
		karl.kunisch@uni-graz.at}).}
	\maketitle
\begin{abstract}
%\lipsum[1]	
Policy iteration is a widely used technique to solve the
Hamilton Jacobi Bellman (HJB) equation, which arises  from nonlinear optimal feedback control theory. Its convergence analysis has attracted much attention in the unconstrained case. 
Here we analyze  the case with control constraints both for the HJB equations which arise in deterministic and in stochastic control cases. The linear equations in each iteration step are solved by an implicit upwind scheme. 
 Numerical examples are conducted to solve the HJB equation with control constraints and comparisons are shown with the unconstrained cases.
%s	
\end{abstract}

	
%%%%%%%%%%%%%%%%%%%%%%%%%%%%%%%%%%%%%%%%%%%%%%%%%%%%%
\section{Introduction}\label{intro}
%Stabilization and feedback control
Stabilizability is one of the major objectives in the area of optimal control theory. Since  the open loop control depends only on the time and the initial state, so if the initial state is changed, unfortunately control needs to be recomputed again. In many physical situation, we are particularly interested in seeking a control law which depends on the  state and which can deal additional external perturbation or model errors. When the dynamical systems is linear with unconstrained control and corresponding cost functional or performance index is quadratic, then the associated closed loop or feedback control law which minimizes the cost functional  is obtained by the so called Riccati equation. Now when the dynamics is nonlinear, one popular approach is to linearize the dynamics and obtain the Riccati based controller to apply for the original nonlinear system, see e.g. \cite{L69}.

If the optimal feedback cannot be obtained by LQR theory, then it can be approached by the verification theorem and  the value function, which in turn is a solution of the Hamilton Jacobi Bellman (HJB) equation associated to the optimal control problem, see e.g.  \cite{FR12}.  But in most of the situations, it is very difficult to obtain the exact value function and thus one has to resort to iterative techniques.  One possible approach utilizes the so called value iteration. Here we are interested in more efficient technique,
known as policy iteration. Discrete counterpart of policy iteration is also known as Howards' algorithm \cite{H60}, compare also \cite{BMZ09}. Policy iteration can be interpreted as a Newton method applied to the HJB equation. Hence, using the policy iteration HJB equation reduces to a sequence of linearized  HJB equation, which for historical reasons are called 'generalized HJB equations'.  The policy iteration requires as initialization  a stabilizing control. If such a control is not available, then one can  use discounted path following policy iteration, see e.g. \cite{kk2018}. The focus of the present paper is the analysis of the policy iteration in the presence of control constraints. To the authors' knowledge, such an analysis is not available in the literature. 

Let us next mentions very selectively, some of the literature on the numerical solution of the control-HJB equations. Specific comments for  the constraint case,  are given further below.
Finite difference methods and vanishing viscosity method are developed in the work of Crandall and Lions \cite{CL84}. Other notable techniques are  finite difference methods \cite{CDI84}, semi-Lagrangian schemes \cite{AFK15, FF14}, the finite element method \cite{GR85}, filtered schemes \cite{FS}, domain decomposition methods \cite{CCFP12}, and level set methods \cite{OF03}. For an overview we refer to \cite{FF16}.
Policy iteration algorithms are developed  in \cite{LL67, SL79, BST98, BM98}.
If the dynamics are given by a partial differential equation (PDE), then the corresponding HJB equations become infinite dimensional. Applying grid based scheme to convert PDE dynamics to  ODE dynamics leads to a high dimensional HJB equation. This phenomenon is known as the curse of dimensionality. In the literature there are different techniques to tackle this situation. Related recent works include, polynomial approximation \cite{kk2018}, deep neural technique \cite{KW20}, tensor calculus \cite{DKK19}, Taylor series expansions \cite{BKP19} and graph-tree structures \cite{AFS18}.

%HJB approximation second order
Iteration in policy space for second order HJB PDEs arising in stochastic control is discussed in \cite{WS02}.  Tensor calculus technique is used in \cite{YPL} where the associated HJB equation becomes linear after using some exponential transformation and scaling factor. For an overview for numerical approximations to stochastic HJB equations, we refer to \cite{KD01}. To solve the continuous time stochastic control problem by the Markov chain approximation method (MCAM) approach, the diffusion operator is approximated by a finite dimensional Markov decision process, and further solved by the policy iteration algorithm.
For the connection between the finite difference method and the MCAM  approach for the second order  HJB equation, see \cite{BZ03}.
%with control constraint
Let us next recall  contributions for the case with control constraints.
Regularity of the value function in the presence of control constraints has been discussed in \cite{G05, G07} for linear dynamics and in \cite{CF13}  for nonlinear systems. In \cite{L98} non-quadratic penalty functionals were introduced to approximately  include  control constraints. Later in \cite{AL05} such functionals were  discussed  in the context of policy iteration for HJB equations.

Convergence of the policy iteration without constraints  has been investigated in earlier work, see in particular \cite{SL79, WS02}.
In our work, the control constraints are realized exactly by means of a projection operator.
This approach has  been used earlier only  for numerical purposes  in the context of the value iteration, see \cite{GK16} and \cite{KKZ14}.  In our work  we  prove the convergence of the policy iteration with control constraints for both the first and the second order HJB-PDEs. For numerical experiments we use an implicit  upwind scheme as proposed in \cite{Aal17}.

The rest of the paper is organized as follows. Section $2$ provides  a convergence of the policy iteration  for the first order HJB equation in the presence of control constraints. Section $3$ establishes corresponding convergence result  for the second order HJB equation with control constraints. Numerical tests are presented in Section $4$.
Finally in Section $5$ concluding remarks are given.

\section{Nonlinear $\cH_2$ feedback control problem subject to deterministic system}
\subsection{First order Hamilton-Jacobi-Bellman equation}
We consider the following infinite horizon optimal control problem
\begin{align}\label{cost}
\underset{u(\cdot)\in\cU}{\min}\;\cJ(x,u(\cdot)):=\int\limits_0^\infty \Big(\ell(y(t))+\|u(t)\|_R^2\Big)\, dt,
\end{align}
subject to the nonlinear deterministic dynamical constraint
\begin{equation}\label{eq1.1}
\dot y(t)= f(y(t))+g(y)u(t)\,,\quad y(0)=x,
\end{equation}
where $y(t)=(y_1(t),\ldots,y_d(t))^t\in  \R^d$ is the state vector, and $u(\cdot)\in\cU$ is the control input with $\cU=\{u(t):\, \R_+\rightarrow U\subset\R^m\}$. Further $\ell(y)>0$, for $y\neq 0$, is the state running cost, and $\|u\|_R^2=u^tRu$, represents the control cost,  with $R\in\R^{m\times m},\,R>0$  a positive definite matrix. Throughout we assume that the dynamics $f$ and $g$, as well as $\ell$ are Lipschitz continuous on $\R^d$, and that $f(0)=0$ and $\ell(0)=0$. We also require that $g$ is globally bounded on $R^d$. This set-up relates to our aim  of asymptotic stabilization to the origin by means of the control $u$. We shall concentrate on the case where the initial conditions are
chosen from a domain $\Omega\subset \R^d$ containing the origin in its interior.

The specificity of this work relates to controls which need to obey constraints $u(t)\in U$, where $U$ is a closed convex set containing $0$ in $\R^m$.
As a special case we mention bilateral point-wise constraints of the form
\begin{equation}\label{eq:kk1}
U=\{u \, | \,\alpha\leq u \le \beta  \},
\end{equation}
where $\alpha=(\alpha_1,\ldots,\alpha_m)\in \R^m$, $\beta=(\beta_1,\ldots,\beta_m)\in\R^m$, $\alpha\leq0,$ $\beta\geq 0$, and the inequalities act coordinate-wise.\\

 The optimal value function associated to \eqref{cost}-\eqref{eq1.1} is given by
\[V(x)=\underset{u(\cdot)\in\cU}{\min} \cJ(x,u(\cdot)),\]
where $x\in \Omega$. Here and throughout the paper we assume that for every $x\in \Omega$ a solution to \eqref{cost} exists. Thus, implicitly we also assume that the existence of a control $u\in L^2(0,\infty;\R^m)$ such that \eqref{eq1.1} admits a solution  $y\in W^{1,2}(0,\infty;\R^d)$. If required by the context, we shall indicate the dependence of $y$ on $x\in \Omega$ or $u \in \mathcal{U}$, by writing $y(\cdot;x)$, respectively $y(\cdot;u)$.

In case $V\in C^1(\R^d)$ it satisfies the Hamilton-Jacobi-Bellman (HJB) equation
\begin{equation}\label{hjb}
\underset{u\in U}{\inf}\{ \nabla V(x)^t(f(x)+g(x) u)+ \ell(x)+\|u\|_R^2\}=0\,,\quad V(0)=0\,,
\end{equation}
with $\nabla V(x)=(\px{1}V,\ldots,\px{d}V)^t$. Otherwise, sufficient conditions which guarantee that $V$ is the unique viscosity solution to \eqref{hjb}  are well-investigated, see for instance \cite[Chapter III]{BCD97}.

In the unconstrained case, with $U=\R^m$ the value function is always differentiable  (see \cite[pg. 80]{BCD97}).  If $\ell=x^tQx$, with $Q$
a positive definite matrix, and linear control dynamics of the form $f(y)+g(y)u=Ay+Bu$, the value function is differentiable provided that $U$ is nonempty polyhedral (finite horizon) or closed convex set with $0\in int\hspace{0.1cm} U$ (infinite horizon), see \cite{G05, G07}. Sufficient conditions for (local) differentiability of the value function associated to  finite horizon optimal control problems with  nonlinear  dynamics and control constraints  have been  obtained in \cite{CF13}. For our analysis, we assume that

\begin{assumption}\label{ass1}
The value function satisfies  $V\in C^1(\R^d)$. Moreover it is radially unbounded, i.e. $\lim_{\|x\|\to \infty}V(x) =\infty$.
\end{assumption}

With Assumption  \ref{ass1} holding, the verification theorem, see eg. \cite[Chapter 1]{FS06} implies  that an optimal control in feedback form is given as the minimizer in \eqref{hjb} and thus
\begin{equation}\label{proj}
 u^*(x)=\cP_{U}\Big(-\frac{1}{2}R^{-1}g(x)^t\nabla V(x)\Big),
\end{equation}
where $\cP_{U}$ is the orthogonal projection in the $R$-weighted inner product on $\R^m$ onto  $U$, i.e.
$$ (R^{-1}g(x)^t\nabla V(x) +2u^*, u-u^*)_R \ge 0 \text{ for all } u \in U.$$
Alternatively $u^*$ can be expressed as
\begin{equation*}
 u^*(x)=R^{-\frac{1}{2}}\cP_{\bar U}\Big(-\frac{1}{2}R^{-\frac{1}{2}}g(x)^t\nabla V(x)\Big),
\end{equation*}
where $\cP_{\bar U}$ is the orthogonal projection in $\R^m$, with Euclidean inner product,  onto $\bar U=R^{\frac{1}{2}}U$. For the particular case of \eqref{eq:kk1} we have
\begin{equation*}
 u^*(x)=\cP_{U}\Big(-\frac{1}{2}R^{-1}g(x)^t\nabla V(x)\Big)=\min\Big\{\beta,\max\{\alpha,-\frac{1}{2}R^{-1}g(x)^t\nabla V(x)\}\Big\},
\end{equation*}
where  the $\max$ and $\min$ operations operate coordinate-wise. It is common practice to refer to the open loop control as in \eqref{cost} and to the closed loop control as in \eqref{proj} by the same letter.

For the unconstrained case the corresponding control law reduces to
 $u^*(x)=-\frac{1}{2}R^{-1}g(x)^t\nabla V(x)$.
 Using the optimal control law $u(x)=\cP_{U}\Big(-\frac{1}{2}R^{-1}g(x)^t\nabla V(x)\Big)$ in \eqref{hjb}, (where we now drop the superscript $^*$) we obtain the equivalent form of the HJB equation as
\begin{equation}\label{hjb1}
\nabla V(x)^t\Bigg(f(x)+g(x)\cP_{U}\Big(-\frac{1}{2}R^{-1}g(x)^t\nabla V(x)\Big)\Bigg)+\ell(x)+\norm{\cP_{U}\Big(-\frac{1}{2}R^{-1} g(x)^t \nabla V(x)\Big)}^2_R=0.
\end{equation}
We shall require the notion of  generalized Hamilton-Jacobi-Bellman (GHJB) equations (see e.g. \cite{SL79, BST7}), given as follows:
 \begin{equation}\label{gHJB}
 GHJB(V,\nabla V;u):=\nabla V^t\Big(f+gu\Big)+\Big(\ell+\norm{u}^2_{R}\Big)=0,\qquad V(0)=0,
 \end{equation}
Here $f,g$ and $u$ are considered as functions of $x\in\R^d$.

\begin{remark}
Concerning the  existence of a solution in the unconstrained case to the  Zubov type equation \eqref{gHJB},
we refer to  \cite[Lemma 2.1, Theorem 1.1]{L69}
where it is shown in the analytic case   that if $(\frac{\partial f}{\partial y}(0),g(0))$ is a stabilizable pair, then \eqref{eq2.3} admits a unique locally positive definite solution $V(x)$, which is locally analytic at the origin.
\end{remark}

For solving \eqref{hjb1} we analyse the  following iterative scheme, which is referred to as policy iteration or Howards' algorithm, see e.g.  \cite{SL79, AFK15, BST7, BMZ09}, where the case without control constraints is treated.
We require the notion of admissible controls, a concept introduced in \cite{LL67, BST7}.
\begin{definition}\label{def1}(Admissible Controls). A measurable function $u:\R^d\to U\subset \R^m$ is called  admissible with respect to  $\Omega$, denoted by  $u \in \mathcal{A}(\Omega)$, if
	\begin{itemize}
		\item[(i)] $u$ is continuous on $\R^m$,
		\item[(ii)] $u(0)=0$,
		\item[(iii)] $u$ stabilizes \eqref{eq1.1} on $\Omega$, i.e. $\lim_{t\to\infty}y(t;u)=0,   \quad \forall x\in \Omega$.
		\item[(iv)] $\int\limits_0^\infty \Big(\ell(y(t;u))+\|u(y(t;u))\|_R^2\Big)\, dt < \infty,  \quad \forall x\in \Omega$.
	\end{itemize}
\end{definition}
Here  $y(t;u)$ denotes the solution to \eqref{eq1.1}, where $x\in \Omega$, and with control in feedback form $u(t) = u(y(t;u))$.  As in \eqref{cost} the value of the cost in (iv) associated to the closed loop control is denoted by $\cJ(x,u)  = \int\limits_0^\infty (\ell(y(t;u))+\|u(y(t;u))\|_R^2)\, dt$.  In general we cannot guarantee that the controlled trajectories $t\to y(t;u(t))$ remain in $\Omega$ for all $t$. For this reason we demand continuity and differentiability properties of $u$ and $V$ on all of $\R^d$. Under additional assumptions we could introduce a set $\tilde \Omega$ with $\Omega \subsetneq \tilde \Omega \subsetneq \R^d$ with the property that
$\{y(t,u) : t\in[0,\infty), u \in \mathcal{A}(\Omega) \} \subset\tilde \Omega$
and demanding the regularity properties  of $u$ and $V$ on $\tilde \Omega$ only. We shall not pursue such a  set-up here.

We are now prepared the present the algorithm.
\begin{algorithm}[!h]
	\begin{algorithmic}
		\STATE{{\bf Input}: Let $u^{(0)}:\R^d\mapsto U\subset \R^m$ be an admissible control law for the  dynamics \eqref{eq1.1} and let   $\epsilon>0$  be a given tolerance.}
		\STATE{\bf While $\sup_{x\in \Omega}|u^{(i+1)}(x)-u^{(i)}(x)|\geq\epsilon$, \;}
		\STATE{\bf Solve for $V^{(i)}\in C^1(\R^d):$
			\begin{align}\label{eq2.3}
			\nabla V^{(i)}(x)^t\big(f(x)+g(x) u^{(i)}(x)\big)+ \ell(x)+\|u^{(i)}(x)\|^2_R=0\, \mbox{ with } V^{(i)}(0)=0\,,
			\end{align}}
		\STATE{\bf Update the Control:
			\begin{align}\label{eq2.3a}
			u^{(i+1)}(x)=\cP_{U}\Big(-\frac{1}{2}R^{-1}g(x)^t\nabla V^{(i)}(x)\Big)\,.
			\end{align}
			End}
		\caption{Continuous policy iteration algorithm for first order HJB equation}\label{alg:polit1}
	\end{algorithmic}
\end{algorithm}

Note that \eqref{eq2.3} can equivalently be expressed as $GHJB(V^{(i)},\nabla V^{(i)};u^{(i)})=0,\,  V^{(i)}(0)=0$.

\begin{lemma}\label{lm1}
Assume that $u(\cdot)$ is an admissible feedback control with respect to  $\Omega$. If there exists a function  $V(\cdot;u)\in C^1(\R^d)$ satisfying
\begin{equation}\label{eqx1.2}
GHJB(V,\nabla V;u)=\nabla V(x;u)^t(f(x)+g(x)u(x))+\ell(x)+\norm{u(x)}^2_R=0, \quad V(0;u)=0,
\end{equation}
then $V(x;u)=\cJ(x,u)$ for all $ x\in \Omega$.
Moreover, the optimal control law $u^*(x)$ is admissible on $\R^d$,  the optimal value function $V(x)$ satisfies $V(x)=\cJ(x,u^*)$, and $0<V(x)\leq V(x;u)$.
\end{lemma}
\begin{proof}
The proof takes advantage, in part, from the verification of an analogous result in  the unconstrained, see eg. \cite{SL79}, where a finite horizon problem is treated.  Let $x\in \Omega$ be arbitrary and fixed, and choose any $T>0$.
Then  we have
\begin{equation}\label{eqaux1}
V(y(T;u);u)-V(x;u)=\int_{0}^{T}\frac{d}{dt}  V( y(t;u))\; dt.
\end{equation}
Since $\lim_{T\to\infty}y(T;u)=0$ by  (iii), and  due to V(0;u)=0, we can take the limit $T\to \infty$ in this equation to obtain
\begin{align}\label{eq1.2}
 V(y(\infty;u);u)-V(x;u)=  -V(x;u)&=\int_{0}^{\infty} \frac{d}{dt} V (y(t;u);u)\; dt\notag\\
 & =\int_{0}^{\infty}\nabla V(y;u)^t(f(y)+g(y) u(y)))\; dt,
\end{align}
where $y=y(t;u)$ on the right hand side of the above equation.
Adding both sides of $\cJ(x,u)=\int_{0}^{\infty} (\ell(y)+\norm{u}^2_R)\; dt$ to the respective sides of \eqref{eq1.2}, we obtain
\begin{align*}
\cJ(x,u)-V(x;u)=\int_{0}^{\infty}\Big(\nabla V(y;u)^t(f(y)+g(y) u(y))+\ell(y)+\norm{u(y)}^2_R\Big) dt,
\end{align*}
and from \eqref{eqx1.2} we conclude  that $V(x;u)=\cJ(x,u)$ for all $x\in \Omega$.

Let us next consider the optimal feedback law  $u^*(x)$. We need to show that it is admissible in the sense of Definition \ref{def1}.
By Assumption \ref{ass1} and \eqref{proj} it is continuous on $\R^d$. Moreover $V(0)=0$, $V(x)> 0$ for all $0\neq x\in \Omega$,  thus $\nabla V(0) = 0$, and consequently $u^*(0)=0$. Here we also use that $0\in U$. Thus (i) and (ii) of Definition \ref{def1} are satisfied for $u^*$ and (iv) follows from our general assumption that \eqref{cost} has a solution for every $x\in \Omega$. To verify (iii), note that from \eqref{eqaux1}, \eqref{proj}, and \eqref{hjb1} we have for every $T>0$
\begin{equation}\label{eqaux2}
V(y(T;u^*)) - V(x)= -\int^T_0 \ell(y(t;u^*)) + \|u^*(y(t;u^*)) \|_R^2 \; dt <0.
\end{equation}
Thus $T \to V(y(T;u^*))$ is strictly monotonically decreasing, (unless $y(T;u^*)=0$ for  some $T$). Thus $\lim_{T\to\infty} V(y(T;u^*))= \epsilon$ for some $\epsilon\ge 0$. If $\lim_{T\to \infty} y(T;u^*) \neq 0$, then there exists $\epsilon >0$ such that $\epsilon \le  V(y(T;u^*)) \le V(x)$ for all $T\ge 0$. Let $S=\{z \in \R^d: \epsilon \le V(z) \le V(x)\}$. Due to continuity of $V$ the set $S$ is closed. The radial unboundedness assumption on $V$ further implies that $S$ is bounded and thus it is compact. Let us set
\begin{equation*}
\dot V(x)=\nabla V(x)^t (f(x)+g(x) \cP_{U}(-\frac{1}{2}R^{-1}g(x)^t\nabla V(x))) \text{ for }x\in\R^d.
\end{equation*}
Then by \eqref{hjb1} we have
\begin{equation*}
\dot V(x)= -\ell(x)-\|{\cP_{U}(-\frac{1}{2}R^{-1} g(x)^t \nabla V(x))}\|^2_R<0.
\end{equation*}
By compactness of $S$ we have $\max_{z\in S} \dot V(z)=:\zeta<0$.
Note that $\{T: y(T;u^*)\}\subset S$. Hence by \eqref{eqaux2} we find
\begin{equation*}
\lim_{T\to \infty} V(y(T;u^*)) - V(x) \le \lim_{T\to \infty} \zeta T,
\end{equation*}
 which is impossible. Hence $\lim_{T\to \infty} y(T;u^*) =0$ and $u^* \in \mathcal{A}(\Omega)$.

  Now we can apply the arguments from the first part of the proof with $u=u^*$ and obtain $V(x)=\cJ(x,u^*) \le \cJ(x,u)=V(x;u)$ for every $u\in\cA(\Omega)$. This concludes the proof.
\end{proof}
\subsection{Convergence of policy iteration}
The previous lemma establishes the fact that the value $\cJ(x,u)$, for a given  admissible control $u$, and  $x\in \Omega$ can  be obtained as the evaluation of the solution of the GHJB equation \eqref{eqx1.2}. In the following lemma we commence the analysis of the convergence of Algorithm \ref{alg:polit1}.
\begin{proposition}\label{lm2.1}
	If $u^{(0)} \in \cA(\Omega)$, then $u^{(i)}\in \cA(\Omega)$ for all $i$.
	Moreover we have $V(x)\leq V^{(i+1)}(x)\leq V^{(i)}(x)$. Further, $\{V^{(i)}(x)\}$ converges from above pointwise to some $\bar V(x) \ge V(x)$ in $\Omega$.
\end{proposition}
\begin{proof}
We proceed by induction. Given $u^{(0)} \in \cA(\Omega)$, we assume that $u^{(i)}\in \cA(\Omega)$,  and  establish that  $u^{(i+1)}\in \cA(\Omega)$. We shall frequently refer to Lemma \ref{lm1} with $V^{(i)}:=V(\cdot;u^i)$ and $V^{(i)}$ as in \eqref{eq2.3}.
In view of \eqref{eq2.3a}
$u^{(i+1)}$ is continuous since $g$ is continuous and $V^{(i)}\in C^1(\Omega)$.
Using Lemma \ref{lm1} we obtain that $V^{(i)}$ is positive definite. Hence  it attains its minimum at the origin, $\nabla V^{(i)}(0)=0$ and consequently $u^{(i+1)}(0)=0$. Thus (i) and (ii) in the definition of admissibility of $u^{(i+1)}$ are established.

Next we take the time-derivative of $t\to V^{(i)}(y(u^{(i+1)})(t)$ where $y(u^{(i+1)})$ is the trajectory corresponding to
\begin{equation*}
\dot y(t)= f(y(t))+g(y)u^{(i+1)}(y(t))\,,\quad y(0)=x.
\end{equation*}
Let us  recall that ${\nabla V^{(i)}(y)}^tf(y)=-{\nabla V^{(i)}(y)}^tg(y)u^{(i)}(y)-\ell(y)-\norm{u^{(i)}(y)}^2_{R}$,
for $y\in \R^d$ and set  $y^{(i+1)}= y(u^{(i+1)})$. Then we  find
\begin{equation}\label{eq:aux3}
\begin{array}l
\frac{d}{dt} V^{(i)}(y^{(i+1)}(t))={\nabla V^{(i)}(y^{(i+1)}) }^t(f(y^{(i+1)})+g(y^{(i+1)})u^{(i+1)}(y^{i+1})\\[1.7ex]
=-\ell(y^{(i+1)} )\!-\!\norm{u^{(i)}(y^{(i+1)})}^2_{R}
+{\nabla V^{(i)}(y^{(i+1)})}^tg(y^{(i+1)}) \big(u^{(i+1)}(y^{(i+1)})-u^{(i)}(y^{(i+1)})\big).
\end{array}
\end{equation}
Throughout the following computation we do not indicate the dependence on $t$ on the right hand side of the equality.
Next we need to rearrange the terms on the right hand side of the  last expression.
For this purpose it will be convenient to introduce $z=\frac{1}{2}R^{-1} g(y^{(i+1)})^t{\nabla V^{(i)}(y^{(i+1)})}$ and observe that $u^{(i+1)}=\cP_{U}(-z)$. We can express the above equality as
\begin{align*}
\frac{d}{dt} V^{(i)}(y^{(i+1)}(t))&= -\ell(y^{(i+1)} )-\norm{u^{(i)}(y^{(i+1)})}^2_{R} + 2 z^t R \cP_{U}(-z) -2 z^t R \, u^{(i)}(y^{(i+1)})\\
&=-\ell(y^{(i+1)}) - \|\cP_{U}(-z) \|_R^2 - \|u^{(i)}(y^{(i+1)}) - \cP_{U}(-z) \|_R^2\\  &\quad+ 2(z+ \cP_{U}(-z))^tR(\cP_{U}(-z)-u^{(i)}(y^{(i+1)})).
\end{align*}
Since $u^{(i)}(y^{(i+1)})\in U$ we obtain   $(z+ \cP_{U}(-z))^tR(\cP_{U}(-z)-u^{(i)}(y^{(i+1)})) \le 0$, and thus

\begin{equation}\label{eq:2.14}
\frac{d}{dt} V^{(i)}(y^{(i+1)}(t))\le -\ell(y^{(i+1)} )-\|u^{(i+1)}(y^{(i+1)})\|^2_{R} -\|u^{(i+1)}(y^{(i+1)}) -u^{(i)}(y^{(i+1)}) \|^2_{R}.
\end{equation}
Hence $t\to  V^{(i)}(y^{(i+1)}(t))$ is strictly monotonically decreasing. As mentioned above, $V^{i}$ is positive definite. With the arguments as in the last part of the proof of Lemma \ref{lm1}  it follows that $\lim_{t\to \infty} y^{(i+1)}(t)=\lim_{t\to \infty}y(t;u^{(i+1)})=0$. Finally \eqref{eq:2.14} implies that
$0\leq \int_{0}^{\infty}\Big(\ell(y(t;u^{(i+1)}))+\|u^{(i+1)}(y^{(i+1)})(t)\|_R^2\Big)\, dt\leq V^{(i)}(x)$. Since $x\in \Omega$ was chosen arbitrarily it follows that $u^{(i+1)}$  defined in \eqref{eq2.3a} is admissible. Lemma \ref{lm1} further implies that $V^{(i+1)}(x) \ge V(x)$ on $\Omega$. \\

Since for each $x\in \Omega$ the trajectory $y^{(i+1)}$  corresponding to $u^{(i+1)}$  and satisfying
\begin{align}\label{eq2.8}
\dot y&= \Big(f(y)+g(y)u^{(i+1)}\Big), \quad y(0)=x,
\end{align}
is asymptotically stable, the difference between  $V^{(i+1)}(x)$ and $V^{(i)}(x)$ can be obtained as
\begin{align*}
V^{(i+1)}(x)-V^{(i)}(x)
&=\int_{0}^{\infty}\Bigg(\Big(\nabla V^{(i)}(y^{(i+1)})^t(f+gu^{(i+1)})\Big)
-\Big({\nabla V^{(i+1)}}^t(f+gu^{(i+1)})\Big)\Bigg)\; dt,
%&=: \int_{0}^{\infty}(A-B)\; dt.
\end{align*}
where $f$ and $g$ are evaluated at $y^{(i+1)}$.
Utilizing the generalized HJB equation \eqref{eq2.3}, we get
${\nabla V^{(i)}(y^{(i+1)})}^t(f+gu^{(i+1)})={\nabla V^{(i)}(y^{(i+1)})}^tg(u^{(i+1)}-u^{(i)})-(\ell+\norm{u^{(i)}}^2_R)$ and ${\nabla V^{(i+1)}(y^{(i+1)})}^t(f+gu^{(i+1)})=-(\ell+\norm{u^{(i+1)}}^2_R)$.
This leads to
\begin{align*}
V^{(i+1)}(x)-V^{(i)}(x)= \int^\infty_0 \big(  \|u^{(i+1)}\|^2_R-\|u^{(i)}\|^2_R +{\nabla V^{(i)}(y^{(i+1)})}^tg(u^{(i+1)}-u^{(i)})\big)  dt.
\end{align*}
The last two terms in the above integrand appeared in \eqref{eq:aux3} and were estimated in the subsequent steps. We can reuse this estimate and obtain
\begin{align*}
V^{(i+1)}(x)-V^{(i)}(x)\le - \int^\infty_0  \|u^{(i+1)}-u^{(i)}\|^2_R \,dt \le 0.
\end{align*}

Hence,  $\{V^{(i)}\}$ is a monotonically decreasing sequence which is bounded below by the optimal
value function $V$, see Lemma \ref{lm1}. Since $\{V^{(i)}\}$ is a monotonically decreasing sequence and bounded below by $V$, it converges pointwise to some $\bar V \ge V$.
\end{proof}
To show  convergence of $u^{(i)}$ to $\bar u :  = \cP_{U}(-\frac{1}{2}R^{-1}g(x)^t\nabla \bar V(x))$, additional assumptions are needed. This is considered in  the following proposition.  In the literature, for the unconstrained case, one can  find the statement that, based on Dini's theorem,   the monotonically convergent sequence  $\{V^{(i)}\}$ converges uniformly to $\bar V$, if $\Omega$ is compact. This, however, only holds true, once it is argued that $\bar V$ is continuous.
For the following it will be useful to recall that $C^m(\bar\Omega)$, $m\in \N_0$, consists of all functions $\phi\in C^m(\Omega)$ such that $D^{\alpha}\phi$ is bounded and uniformly continuous on $\Omega$ for all multi-index $\alpha$ with $0\leq\alpha\leq m$, see e.g. \cite[pg. 10]{adams2003}.
\begin{proposition}\label{lm3}
If $\Omega$ is bounded, and further  $\{V^{(i)}\}\subset C^1(\Omega)$ satisfy \eqref{eq2.3}, $\bar V\in C^1(\Omega)\cap C(\bar \Omega) $, and $\{\nabla V^{(i)}\}$ is equicontinuous in $\Omega$, then $\{\nabla V^{(i)}\}$ converges pointwise to $\nabla\bar V$ in $\Omega$,
and $\bar V$ satisfies  the HJB equation \eqref{hjb1} for all $x\in\Omega$ with $\bar V(x)=V(x)$ for all $x\in \Omega$.
\end{proposition}
\begin{proof} Let $x\in \Omega$ and $\epsilon>0$ be arbitrary. Denote by $e_k$ the $k$-th unit vectors, $k=1,\hdots,d,$ and choose $\delta_1>0$  such that $S_{x}=\{x+h:|h| \le \delta_1\}\;\subset\Omega$.
 By continuity of $\nabla \bar V$ and equicontinuity of $\nabla  V^{(i)}$ in $ \Omega$, there exists $\delta\in(0,\delta_1)$ such
\begin{align}\label{eq1}
\|\nabla\bar V(x)-\nabla\bar V(x+h)\|<\frac{\epsilon}{3} \quad \text{and}\quad \|\nabla V^{(i)}(x)-\nabla V^{(i)}(x+h)\|<\frac{\epsilon}{3},
\end{align}	
with for all $h$ with	$|h|\leq \delta$, $k=\{1,\hdots,d\}$, $i=1,2,\hdots$.

By assumption $\Omega$ is bounded and thus $\bar \Omega$ is compact. Hence by Dini's theorem $V^{(i)}$ converges to $\bar V$ uniformly  on $\bar \Omega$. Here we use the assumption that  $\bar V\in C(\bar \Omega)$.  We can now choose
 $\bar i>0$	such that
\begin{align}\label{eq2}
|\bar V(y)- V^{(i)}(y)|\leq \frac{\epsilon\delta}{6}\quad \forall\, y\in \Omega,\quad \text{and } \forall \, i\geq \bar i.
\end{align}
We have
	\begin{align*}
&\partial_{x_k}\bar V(x)-\partial_{x_k} V^{(i)}(x)\\
&=\Big(\partial_{x_k}\bar V(x)-\frac{1}{\delta}\big(\bar V(x+e_k\delta)-\bar V(x)\big)\Big)+\frac{1}{\delta}\Big(\big(\bar V(x+e_k\delta)-\bar V(x)\big)-\big( V^{(i)}(x+e_k\delta)- V^{(i)}(x)\big)\Big)\\
&\qquad+\Big(\frac{1}{\delta}\big( V^{(i)}(x+e_k\delta)- V^{(i)}(x)\big)-\partial_{x_k} V^{(i)}(x)\Big)=:I_1+I_2+I_3.
	\end{align*}
	We estimate $I_1$ and $I_3$ by using \eqref{eq1} as
\begin{align*}
|I_1|=\frac{1}{\delta}|\int_{0}^{1}\Big(\partial_{x_k}\bar V(x)-\partial_{x_k}\bar V(x+e_k\sigma\delta)\Big) \delta\;d\sigma|\leq \frac{\epsilon}{3},
\end{align*}
and	
\begin{align*}
|I_3|=\frac{1}{\delta}|\int_{0}^{1}\Big(\partial_{x_k} V^{(i)}(x+e_k\sigma\delta)-\partial_{x_k} V^{(i)}(x)\Big) \delta\;d\sigma|\leq \frac{\epsilon}{3}.
\end{align*}
We estimate $I_2$ by using \eqref{eq2}
$$|I_2|\leq \frac{\epsilon}{3},$$ and combining with the estimates for $I_1$ and $I_3$, we obtain
\begin{equation}\label{eq3}
\|\nabla\bar V(x)-\nabla V^{(i)}(x)\|\leq \epsilon\sqrt d.
\end{equation}
Since $x\in \Omega$ was arbitrary this implies that
$$\nabla V^{i}\to \nabla \bar V \quad\text{pointwise in} \quad \Omega.$$
It then follows from \eqref{eq2.3a} that
$$\lim_{i\to\infty}\cP_{U}\Big(-\frac{1}{2}R^{-1}g(x)^t\nabla V^{(i)}(x)\Big)=\lim_{i\to\infty}u^{(i+1)}(x)=\cP_{U}\Big(-\frac{1}{2}R^{-1}g(x)^t\nabla\bar V(x)\Big)=:\bar u(x), \hspace{0.1cm} \text{in}\hspace{0.1cm}\Omega,$$
and by \eqref{eq2.3}
\begin{equation*}
\nabla\bar V(x)^t(f(x)+g\bar u)+\ell(x)+\norm{\bar u}^2_R=0, \quad \forall x\in \Omega.
\end{equation*}
For uniqueness of value function $\bar V=V$, we refer to \cite[pg. 86]{FS06} and \cite[Chapter III]{BCD97}.
\end{proof}

While the assumptions of Proposition \ref{lm3} describe a sufficient conditions for
pointwise convergence of $\nabla V^{(i)}$,
is appears to be a challenging open issue to  check them in practice.
%
%
%

%\end{document}
\section{Nonlinear $\cH_2$ control subject to stochastic system}
\subsection{Second order Hamilton-Jacobi-Bellman equation}
Here we consider the stochastic  infinite horizon  optimal control problem
\begin{align}\label{scost}
\underset{u(\cdot)\in\cU}{\min}\;\cJ(x,u(\cdot)):=\E\int\limits_0^\infty \Big(\ell(y(t))+\norm{u(t)}^2_{R}\Big)\, dt,
\end{align}
subject to the nonlinear stochastic dynamical constraint
 \begin{align}\label{eq2.5}
d y(t)&= \Big(f(y(t))+g(y)u(t)\Big)dt+g_1(y)\hspace{0.1cm}dW \,,\quad y(0)=x\in \R^d,
\end{align}
where $y(t)\in \R^d$ is the state vector, $W(t)\in \R^k$
is a standard multi-dimensional separable Wiener process
 defined on a complete probability space, and the control input $u(\cdot)\in \cU$  is an adapted process with respect to a natural filter.
The functions $f,\,g,$ and $g_1$  are assumed to be Lipschitz continuous on $\R^d$,  and satisfy $f(0)=0$, $g(0)=0$ and $g_1(0)=0$.
For  details about existence and uniqueness for \eqref{eq2.5}, we refer to  e.g. \cite[Chapter 2]{M07}.
  As in the previous section $\Omega$ denotes a domain in $\R^d$ containing the origin. \\

  The
value function $V(x)=\underset{u(\cdot)\in U}{\inf}\cJ(x,u(\cdot))$  is assumed to be  $C^2$- smooth over $\R^d$.
 Then the HJB equation corresponding to \eqref{scost} and \eqref{eq2.5} is of the form
 \begin{equation}\label{hjbs}
 \underset{u\in U}{\min}\Big\{ \nabla V(x)^t(f(x)+g(x) u)+\frac{1}{2}Tr[g_1(x)^t\frac{\partial^2 V(x)}{\partial x^2}g_1(x)]+\Big( \ell(x)+\norm{u}^2_R\Big)\Big\}=0\,,\quad V(0)=0\,.
 \end{equation}
From the verification theorem \cite[Theorem 3.1]{FS06}, the explicit minimizer $u^*$ of \eqref{hjbs} is given by
 \begin{align}\label{optc}
 u^*(x)&=\underset{u\in U}{argmin}\Big\{ \nabla V(x)^t(f(x)+g u)+\frac{1}{2}Tr[g_1(x)^t\frac{\partial^2 V(x)}{\partial x^2}g_1(x)]+\Big( \ell(x)+\norm{u}^2_R\Big)\Big\}\notag\\
 &=\cP_U\Big(-\frac{1}{2}R^{-1} g(x)^t \nabla V(x)\Big)\,,
 \end{align}
 where the  projection $\cP_{U}$ is defined as in \eqref{proj}.
 The infinitesimal differential generator of the stochastic system \eqref{eq2.5} for the  optimal pair $(u,V)$ (dropping the superscript *) is denoted by
 $$\cL_uV(x)=\nabla V(x)^t\Big(f(x)+g(x)u(x)\Big)+\frac{1}{2}Tr[g_1(x)^t\frac{\partial^2 V(x)}{\partial x^2}g_1(x)].$$
 %To compare between HJB and HJI, we can take $\ell=h^th$.
 Using the optimal control law \eqref{optc} in \eqref{hjbs}, we obtain an equivalent form of the HJB equation
 \begin{align}\label{hjb2}
 \nabla V(x)^t&\Bigg(f(x)+g(x)\cP_U\Big(-\frac{1}{2}R^{-1}g(x)^t\nabla V(x)\Big)\Bigg)+\ell(x)\notag\\
 &+\norm{\cP_U\Big(-\frac{1}{2}R^{-1} g(x)^t \nabla V(x)\Big)}^2_R+\frac{1}{2}Tr[g_1(x)^t\frac{\partial^2 V(x)}{\partial x^2}g_1(x)]= 0\,.
 \end{align}
 For the unconstrained control, the above HJB equation becomes
 \begin{equation}\label{hjbun2}
 \nabla V(x)^t f(x)-\frac14\nabla V(x)^tg(x)R^{-1} g(x)^t \nabla V(x)+\frac{1}{2}Tr[g_1(x)^t\frac{\partial^2 V(x)}{\partial x^2}g_1(x)]+\ell(x)=0\,.
 \end{equation}
 The associated generalized second order  GHJB equation is of the form
 \begin{align}
 GHJB(V,\nabla V,\nabla^2V;u)&=0,\qquad V(0)=0,\label{gHJB1} \text{where}\\
 GHJB(V,\nabla V,\nabla^2V;u)&:=\nabla V^t\Big(f+gu\Big)+\frac{1}{2}Tr[g_1^t\nabla^2Vg_1]+\Big(\ell+\norm{u}^2_{R}\Big)\notag.
 \end{align}
We next define the admissible feedback controls with respect to stochastic set up.
 \begin{definition}\label{def2}(Admissible Controls). A  control $u$ is defined to be admissible with respect to \eqref{scost}, denoted by  $u \in \mathcal{A}(\Omega)$ if
 	\begin{itemize}
 		\item[(i)] $u$ is continuous on $\R^m$,
 		\item[(ii)] $u(0)=0$,
 		\item[(iii)] $u$ stabilizes \eqref{eq2.5} on $\Omega$ stochastically, i.e. $P(\lim_{t\to\infty}y(t;u)=0)=1,   \quad  \forall \ x\in \Omega$ i.e. when $\E(\lim_{t\to\infty}y(t;u))=0$.
 		\item[(iv)] $\E\int\limits_0^\infty \Big(\ell(y(t;u))+\|u(y(t;u))\|_R^2\Big)\, dt < \infty,  \quad \forall \; x\in \Omega$.
 	\end{itemize}
 \end{definition}
In Algorithm \ref{alg:sga2}, the policy iteration for the second order HJB equation is documented, see \cite{WS02}.

 \begin{algorithm}[!h]
 %\begin{algorithm}[H]
 	Let $u^{(0)}$ be an initial stabilizing control law in the stability region $\Omega$ for system \eqref{eq2.5}.
 	\begin{algorithmic}
 		\STATE{\bf For $i=0:\infty$\;}
 		%\WHILE{$error>tol$}
 		\STATE{\bf While $\sup_{x\in\Omega}|u^{(i+1)}(x)-u^{(i)}(x)|\geq \text{tolerance}$, \;}
 		\STATE{\bf Solve for $V^{(i)}$:
 			\begin{align}\label{eq2.13}
 		\cL_{u^{(i)}(x)}V^{(i)}(x)+\ell(x)+&\norm{u^{(i)}(x)}^2_{R}=0\,,\, V^{(i)}(0)=0,
 			\end{align}
 		}
 		
 		\bf {End
 			\STATE{Update the Control:
	\begin{equation}\label{eqx}
 		u^{(i+1)}(x)=\cP_U \Big(-\frac{1}{2}R^{-1}g^t\nabla V^{(i)}(x))\Big)\,.
 \end{equation}
% 				\[
% 				u^{(i+1)}(x)=\cP_U \Big(-\frac{1}{2}R^{-1}g^t\nabla V^{(i)}(x))\Big)\,.
% 				\]
 				\;}
 			End}
 		\caption{Policy iteration algorithm for the second order HJB equation}\label{alg:sga2}
 	\end{algorithmic}
 \end{algorithm}
%
%\\
\begin{lemma}\label{lm2.2}
	Assume that $u(\cdot)$ is an admissible feedback control in $\Omega$. If there exist a function $V(\cdot;u)\in C^2(\mathbb{R}^d)$ satisfying
	\begin{equation}\label{eqx2.2}
	GHJB(V,\nabla V,\nabla^2V;u)=\cL_uV(x)+\ell(x)+\norm{u(x)}^2_R=0, \quad V(0;u)=0,
	\end{equation}
	then $V(x;u)$ is the value function of the system \eqref{eq2.5} and $V(x;u)=J(x,u)$ $\forall x\in \Omega$.
	Moreover, the optimal control law $u^*(x)$  is admissible and the optimal value function $V(x;u^*)$, if it exists and is radial unbounded, satisfies $V(x;u^*)=\cJ(x,u^*)$, and $0<V(x;u^*)\leq V(x;u)$.
\end{lemma}
\begin{proof}
By It\^{o}'s formula \cite[Theorem 6.4]{M07} for any $T>0$
\begin{align*}
&V(y(T);u)-V(x;u)\\
&=\int_{0}^{T}d V(y(t;u);u)\\
&=\int_{0}^{T}\Big(\nabla V(y;u))^t(f(y)+g(y)u(y))+Tr[\frac{1}{2}g_1(y)^t\frac{\partial^2 V(y;u))}{\partial y^2}g_1(y)]\Big)\; dt\\
&\qquad+\int_{0}^{T}\nabla V(y;u))^tg_1(y)\;dW.
\end{align*}
Taking the expectation  with limit as $T\to \infty$ on both sides, the above equation becomes
\begin{align}\label{eq4.1}
	\E\;(\lim_{T\to\infty}V(y(T);u))-V(x;u)
	&=-V(x;u)=\E\int_{0}^{\infty}\cL_{u(y)}V(y)\; dt,
\end{align}
where $\E\int_{0}^{t}\nabla V(y;u)^tg_1(y)\;dW=0$ for each time $t\geq 0$, with $\E(\lim_{T\to\infty}V(y(T);u))=0$.
Adding $\cJ(x,u)=\E\int_{0}^{\infty} (\ell(y)+\norm{u(y)}^2_R)\; dt$ to \eqref{eq4.1}, we obtain
\begin{align*}
	\cJ(x,u)-V(x;u)&=\E\int_{0}^{\infty}\Big(\cL_{u(y)}V(y)+\ell(y)+\norm{u(y)}^2_R\Big) dt.
\end{align*}
Hence $V(x;u)=\cJ(x,u)$ for all $x\in\Omega$.\\
Concerning the admissibility of $u^*$, properties $(i)$, $(ii)$ and $(iv)$ in Definition \ref{def2}, can be shown similarly as in  Lemma \ref{lm1}. Since the stochastic Lyapunov functional $V$ is radially unbounded by assumption, and its differential generator $\cL_{u^*}V$ satisfies $\cL_{u^*}V(x)=-(\ell(x)+\norm{u^*}^2_R)$ as well as $$\E V(y(T;u^*))-V(x)=-\E\int_{0}^{T}\Big(\ell(y)+\norm{u^*}^2_R\Big)\; dt,\,\,\text{for each}\; T>0,$$ one can follow the standard argument given in Theorem \cite[Theorem 2.3-2.4]{M07} to obtain\\
$P(\lim_{t\to\infty}y(t;u^*)=0)=1,   \quad \forall \ x\in \Omega$ and hence $u^*\in \cA(\Omega)$.
\end{proof}
\subsection{Convergence of policy iteration}

Here we turn to the  convergence analysis of the iterates $V^{(i)}$.

\begin{proposition}\label{lm2.3}
	If $u^{(0)} \in \cA(\Omega)$, then $u^{(i)}\in \cA(\Omega)$ for all $i$.
Further we have $V(x)\leq V^{(i+1)}(x)\leq V^{(i)}(x)$,
where $V$ satisfies the HJB equation \eqref{hjb2}. Moreover $\{V^{(i)}\}$ converges pointwise to some $\bar V\geq V$ in $\Omega$.
%Moreover, if $V^{(i+1)}(x)=V^{(i)}(x)$, then $V^{(i)}(x)=V^*(x)$.
\end{proposition}
\begin{proof}
	First we show that if  $u^{(i)}\in \cA(\Omega)$, then $u^{(i+1)}\in \cA(\Omega)$.
	As $g$ is continuous and $V^{(i)}\in C^1(\Omega)$, we have $u^{(i+1)}$ is continuous by \eqref{eqx}.
Since $V^{(i)}$ is positive definite, it attains a minimum at the origin and hence $\nabla V^{(i)}(0)=0$ and consequently $u^{(i+1)}(0)=0$.
The infinitesimal generator 	$\cL_uV(y)$ evaluated along the trajectories generated by $u^{(i+1)}$ becomes
	\begin{align}\label{eqx5.1}
	\cL_uV^{(i)}(y;u^{(i+1)})&={\nabla V^{(i)}}^t(f(y)+g(y)u^{(i+1)})+\frac{1}{2}Tr[g_1(y)^t\frac{\partial^2 V^{(i)}(y)}{\partial y^2}g_1(y)]\notag\\
	&=-\ell(y)-\norm{u^{(i)}}^2_{R}+{\nabla V^{(i)}}^tg(u^{(i+1)}-u^{(i)}),
	\end{align}
	where ${\nabla V^{(i)}}^tf+\frac{1}{2}Tr[g_1^t\frac{\partial^2 V^{(i)}}{\partial y^2}g_1]=-{\nabla V^{(i)}}^tgu^{(i)}-\ell(y)-\norm{u^{(i)}}^2_{R}$.\\
	Now we can calculate $\cL_uV^{(i)}(y;u^{(i+1)})$ as in Proposition \ref{lm2.1} to show that $\cL_uV^{(i)}(y;u^{(i+1)})<0$.
	By the stochastic Lyapunov theorem \cite[Theorem 2.4]{M07}, it follows that $P(\lim_{t\to\infty}y(t;u^{(i+1)})=0)=1$, i.e. $\E(\lim_{t\to\infty}y(t;u^{i+1}))=0$.
	Since $\cL_uV^{(i)}(y;u^{(i+1)})<0$, we find
	\begin{align}
	\E\int_{0}^{T}\Big(\ell(y(t;u^{(i+1)}))+\norm{u^{(i+1)}}^2_R\Big)dt<V^{(i)}(x)\quad \forall\; T>0,
	\end{align}
and thus $u^{(i+1)}$ is admissible.\\
Using that $dy= \Big(f(y)+g(y)u^{(i)}\Big)dt+g_1(y)\hspace{0.1cm}dW$ is asymptotically stable for all $i$, by It\^{o}'s formula along the trajectory
\begin{align}\label{eq2.11}
d y&= \Big(f(y)+g(y)u^{(i+1)}\Big)dt+g_1(y)\hspace{0.1cm}dW,
\end{align}
the difference   $V^{(i+1)}(x)-V^{(i)}(x)$ can be obtained as
\begin{align*}
&V^{(i+1)}(x;u^{(i+1)})-V^{(i)}(x;u^{(i+1)})\\
&=\E\int_{0}^{\infty}\Bigg(\Big({\nabla V^{(i)}}^t(f+gu^{(i+1)})+\frac{1}{2}Tr[g_1^t\frac{\partial^2 V^{(i)}}{\partial y^2}g_1]\Big)\\
&\qquad-\Big({\nabla V^{(i+1)}}^t(f+gu^{(i+1)})+\frac{1}{2}Tr[g_1^t\frac{\partial^2 V^{(i+1)}}{\partial y^2}g_1]\Big)\Bigg) dt =: \E \int_{0}^{\infty}(A-B)\; dt.
\end{align*}
Using the generalized HJB equation \eqref{eq2.13} for two consecutive  iterations $(i)$ and $(i+1)$, we have
$\frac{1}{2}Tr[g_1^t\frac{\partial^2 V^{(i)}}{\partial y^2}g_1]=-{\nabla V^{(i)}}^t(f+gu^{(i)})-(\ell+\norm{u^{(i)}}^2_R)$ and ${\nabla V^{(i+1)}}^t(f+gu^{(i+1)})+\frac{1}{2}Tr[g_1^t\frac{\partial^2 V^{(i+1)}}{\partial y^2}g_1]=-(\ell+\norm{u^{(i+1)}}^2_R)$ respectively.
 Applying $u^{(i+1)}(y)=\cP \Big(-\frac{1}{2}R^{-1}g^t\nabla V^{(i)}(y))\Big)$, we obtain finally
\begin{align*}
A-B=-\Big(\norm{u^{(i)}}^2_R-\norm{u^{(i+1)}}^2_R\Big)+{\nabla V^{(i)}}^tg(u^{(i+1)}-u^{(i)}),
\end{align*}
where $A-B$ can be calculated as in Proposition \ref{lm2.1} to obtain $V^{(i+1)}(x)\leq V^{(i)}(x)$.
Further, $V(x)\leq V^{(i+1)}(x)$. Therefore, $\{V^{(i)}\}$ converges pointwise to some $\bar V$.
\end{proof}
Now, in addition if $\Omega$ is compact and   $\bar V$ is continuous, then by Dini's theorem $\{V^{(i)}\}$ converges uniformly to $\bar V$.
\begin{proposition}\label{lm2.4}
Let $\Omega$ be a bounded domain. If
$\{V^{(i)}\}\in C^2(\Omega)$ satisfy \eqref{eq2.13},  $\bar V\in C^2(\Omega)\cap C^1(\bar\Omega)$, and $\{\nabla^{m}V^{(i)}\}$ is equicontinuous, then $\{\nabla^{m}V^{(i)}\}$ converges pointwise to $\nabla^{m}\bar V$ in $\Omega$, for $m=1$, $2$,
and $\bar V$ satisfies the HJB equation \eqref{hjb2} with $\bar V=V$.
\end{proposition}
\begin{proof}
From Proposition \ref{lm3}, it follows that $\{\nabla V^{(i)}\}$ converges pointwise to $\nabla\bar V$ in $\Omega$ and $\bar u(x)=\cP_U\Big(-\frac{1}{2}R^{-1}g(x)^t\nabla\bar V(x)\Big)$.
Pointwise convergence of $\{\nabla^{2}V^{(i)}\}$ to $\nabla^{2}\bar V$ can be argued as for the first derivatives which was done in the proof of Proposition \ref{lm3}.
We can now pass to the limit $i \to \infty$ in \eqref{eq2.13} to obtain that
$\bar V$ satisfies the  HJB equation \eqref{hjb2}. Concerning the uniqueness of the value function $\bar V=V$,
we refer to e.g.  \cite[pg. 247]{FS06}.
\end{proof}

\section{Numerical examples}
Here we conduct numerical experiments to demonstrate the feasibility of the policy iteration in the presence of constraints and to compare the solutions between constrained and unconstrained formulations.
To describe a setup which is independent of a stabilizing feedback we also introduce a discount factor
$\lambda>0$. Unless specified otherwise we choose $\lambda=0.05$, and we also give results with $\lambda=0$.

\subsection*{Test 1: One dimensional linear equation.}
Consider the following minimization problem $$\underset{u(\cdot)}{\min} \; \cJ(x_0,u(\cdot))=\int_{0}^{\infty}e^{-\lambda t}\Big(\norm{x}^2+\norm{u(x)}^2_{R}\Big) dt$$
subject to the following deterministic dynamics with  control constraint
\begin{equation}\label{eq3.1}
\dot x(t)=0.5x(t)+u; \quad -1\leq u\leq 1, \quad x(0)=x_0;
\end{equation}
and $$\underset{u(\cdot)}{\min} \; \cJ(x_0,u(\cdot))=\E\int_{0}^{\infty}e^{-\lambda t}\Big(\norm{x}^2+\norm{u(x)}^2_{R}\Big) dt$$ subject to the following stochastic system with control constraint
\begin{equation}\label{eq3.2}
 dx(t)=(0.5x(t)+u)\hspace{0.1cm}dt+0.005\hspace{0.1cm}dW; \quad -1\leq u\leq 1, \quad x(0)=x_0,
\end{equation}
where $\lambda> 0$ is the discount factor. We  solve the GHJB equation \eqref{eq2.13} combining  an implicit method and with an upwind scheme over $\Omega=(-2,2)$ as follows
\begin{align}\label{eq3.3}
-\frac{(V^{n}_i-V^{(n-1)}_i)}{dt}&+\lambda V^n_i+\nabla V^n_{iupwind}(f(x_i)+gu^n_i)\notag\\
&+\frac12g_1^t(x_i)D^2V^n_ig_1(x_i)+x_i^2+\norm{u^n_i}^2_R=0,
\end{align}
where the superscript $n$ stands for the iteration loop, the subscript $i$ stands for the mesh point numbering in the state space ($V_i\approx V(x_i)$ for $i=1,\cdots, I \hspace{0.1cm}(\text{total number of grid points})$), $dt$ is the time step, $\Delta x=x_{i+1}-x_{i}$ is the distance between equi-spaced grid points, and $D^2V^n_i=(V^n_{i+1}-2V^n_i+V^n_{i-1})/\Delta x^2$ is the approximation of the second order derivative of the value function. To solve the first order HJB \eqref{eq2.3}, we take $g_1=0$ in \eqref{eq3.3}. \\
For the upwind scheme, we take forward difference $\nabla V_{i,F}=\frac{V_{i+1}-V_i}{\Delta x}$ whenever the drift of the state variable $S_{i,F}=f(x_i)+g(x_i)u_{i,F}>0$, and  backward difference $\nabla V_{i,B}=\frac{V_{i}-V_{i-1}}{\Delta x}$ if the drift $S_{i,B}=f(x_i)+g(x_i)u_{i,B}<0$, and $ D\bar V_i=-2R\bar u_i$ with $\bar u_i=-f(x_i)/g(x_i)$ for $S_{i,F}\leq 0\leq S_{i,B}$. Hence,
$$\nabla V^n_{iupwind}=\nabla V^n_{i,F}\one_{S_{i,F}>0}+\nabla V^n_{i,B}\one_{S_{i,B}<0}+ D\bar V^n_i\one_{S_{i,F}\leq 0\leq S_{i,B}},$$ where $\one$ is the characteristic function. The updated control policy for $(n+1)$ iteration becomes
$$u^{n+1}_i=\cP_{U=[-1,1]}(-\frac{1}{2}R^{-1}g(x_i)^t\nabla V^n_{iupwind}).$$ Note that since we solve the  HJB backward in time, it is necessary to construct the upwind scheme as above which is in a reverse compared  to the form an upwind scheme for
 dynamics forward in time. For more details about upwind schemes for HJB, see \cite{Aal17} including its appendix.

  We choose, $R=0.1$,  $dt=2$, $I=400$ and two different initial conditions $x_0=-1.8$ and $x_0=1$.
Figure \ref{fig:test1.1}(i) depicts the value functions in the unconstrained and the constrained case. As expected they are convex. Outside $[-1,1]$ the two value functions differ significantly. In the constrained case it tends to infinity at $\pm 2$ indicating that for such initial conditions the constrained control cannot stabilize anymore.
Figure \ref{fig:test1.1}(ii), shows the  evolution of the states under the effect of the control as shown in  \ref{fig:test1.1}(iv).  In the transient phase, the decay rate for the constrained control system  is slower than in the unconstrained case, which is the expected behavior. The temporal behavior of the running cost of $\norm{y(t)}$ is documented in Figure \ref{fig:test1.1}(iii). It is clear that the  uncontrolled solution diverges  whereas  $\norm{y(t)}^2$ tends to zero  for the cases with control, both for constrained and unconstrained controls. The pointwise
 error for the HJB equation \eqref{hjb1}  is documented in Figure \ref{fig:test1.1}(v).
 Similar behavior is achieved  for the stochastic dynamics. It is observed from Figure \ref{fig:test1.2}(i) that even with the small intensity of  noise ($g_1=0.005$), we can see the Brownian motion of the state cost for the controlled system. Figure \ref{fig:test1.2}(ii) plots that error for the value function  of the increasing iteration count.

 We have also solved this problem with $\lambda = 0$. For this purpose we initialized the algorithm with the solution of the discounted problem with $\lambda =0.05$. If the  $dt$ was further reduced, then the algorithm converged. Moreover,
the value functionals with $\lambda =0.05$ and $\lambda =0$ are visually indistinguishable.

\begin{figure}[!h]
	\centering
(i){\includegraphics[width=0.45\textwidth]{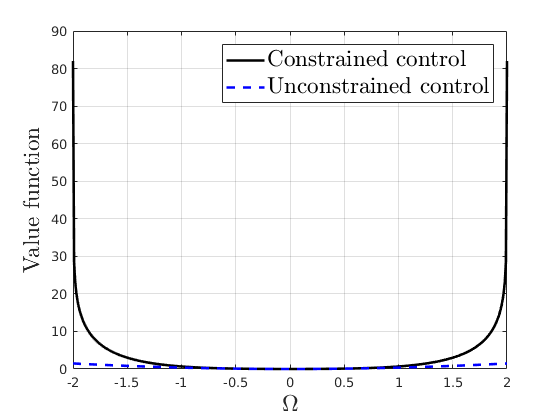}	}
(ii){\includegraphics[width=0.45\textwidth]{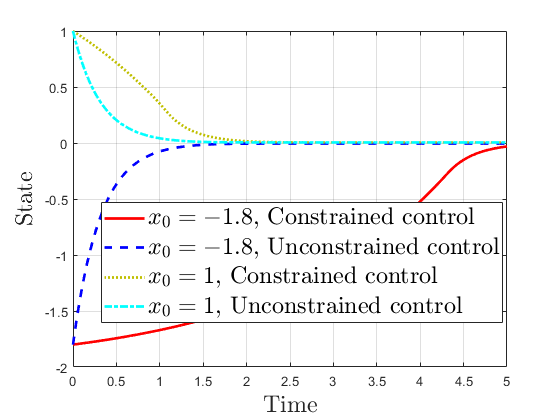}}
 (iii)	{\includegraphics[width=0.45\textwidth]{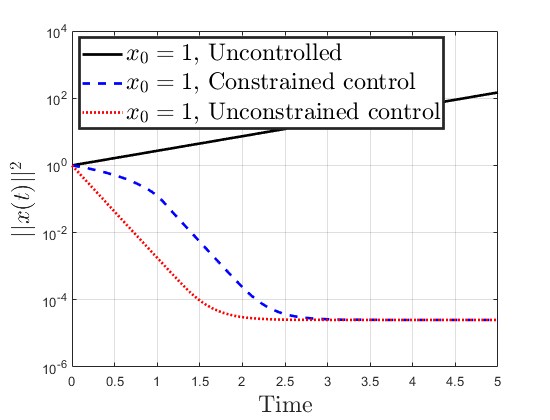}}
(iv){\includegraphics[width=0.45\textwidth]{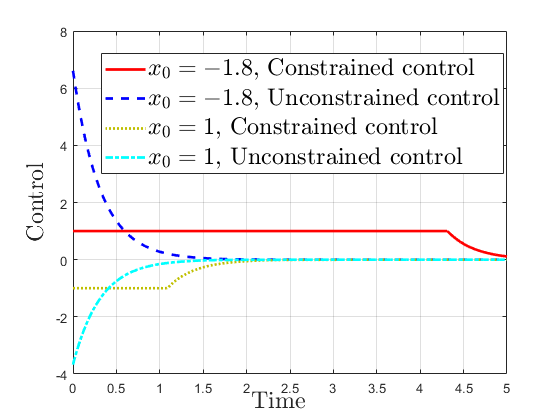}}
(v){\includegraphics[width=0.45\textwidth]{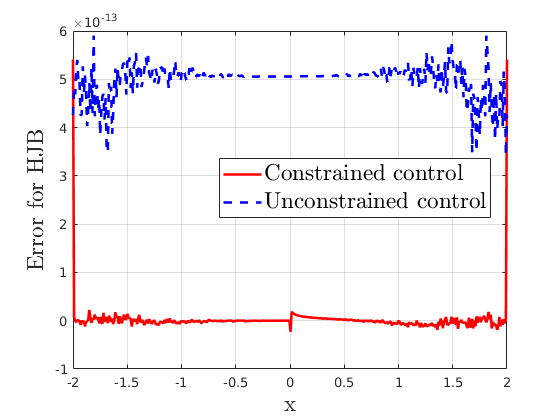}}
	\caption{Test 1: Deterministic case.  {\bf i)} Value function, {\bf ii)} State, {\bf iii)} State cost $\norm{x(t)}^2$, {\bf iv)} Control, {\bf v)} Residue for HJB with discount factor
%{\bf Test 2}: Deterministic case {\bf vi)} Constrained control 
\label{fig:test1.1}.}
\end{figure}
\begin{figure}[!h]
	\centering
%(ii)	\includegraphics[width=0.45\textwidth]{stateex1stoc}
(i)	\includegraphics[width=0.45\textwidth]{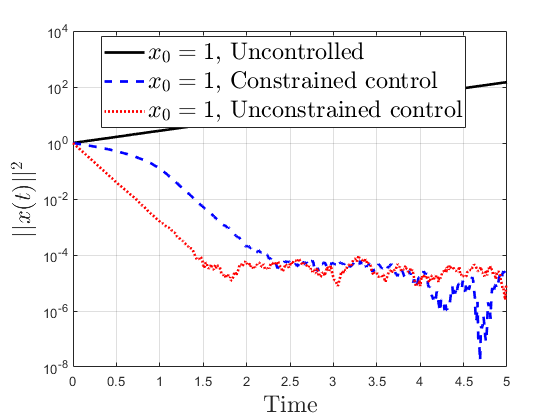}
%(iv)	\includegraphics[width=0.45\textwidth]{controlex1stoc}
%(v)	\includegraphics[width=0.45\textwidth]{errorhjbstoex1}
(ii)\includegraphics[width=0.45\textwidth]{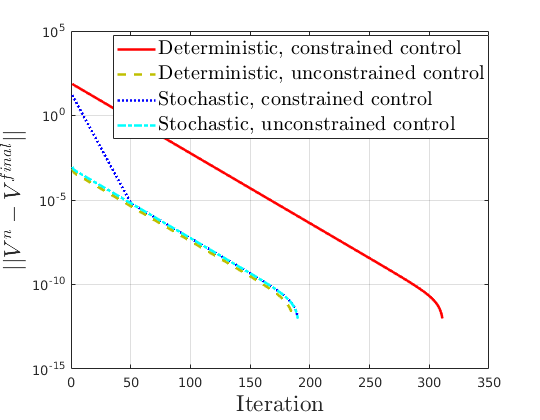}
	\caption{Test 1: Stochastic case.  {\bf i)} State cost $\norm{x(t)}^2$,  {\bf ii)} $||V^{(i)}-V^{Final}||$ for both cases. }\label{fig:test1.2}
\end{figure}

The next example focuses on the exclusion of discount factor once we have proper initialization, which is obtained through solving HJB equation with a discount factor.
\subsection{Test 2: One dimensional nonlinear equation.}
We consider the following infinite horizon problem
$$\underset{u(\cdot)}{\min} \; \cJ(x_0,u(\cdot))=\int_{0}^{\infty}\Big(\norm{x}^2+\norm{u(x)}^2_{R}\Big) dt$$
subject to the following dynamics with  control constraint
\begin{equation}\label{eq5.1}
\dot x(t)=x(t)-x^3(t)+u; \quad -1\leq u\leq 1, \quad x(0)=x_0.
\end{equation}
Here $\Omega=(-2,2)$, $dt=0.001$, $R=0.1$, $I=400$. To obtain a proper initialization $(u^{(0)},V^{(0)})$, we solve \eqref{eq3.3} with discount factor $\lambda=0.05$. Now we can solve \eqref{eq2.3} with the HJB initialization using the aforementioned upwind scheme.

\begin{figure}[!h]
	\centering
	(i)	\includegraphics[width=0.45\textwidth]{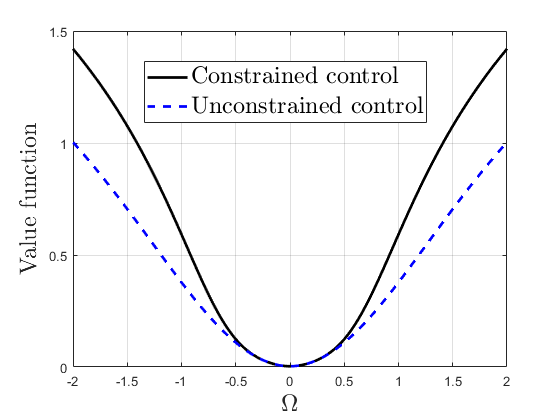}
	(ii)\includegraphics[width=0.45\textwidth]{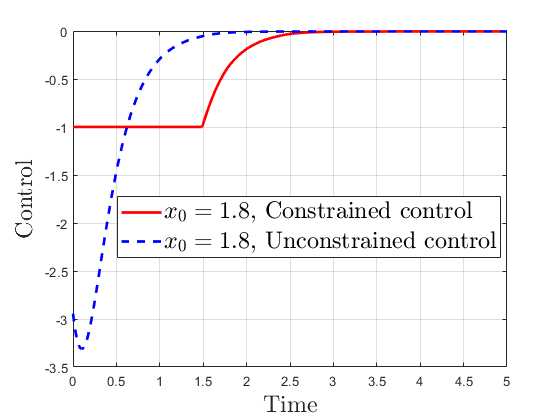}
	(iii)	\includegraphics[width=0.45\textwidth]{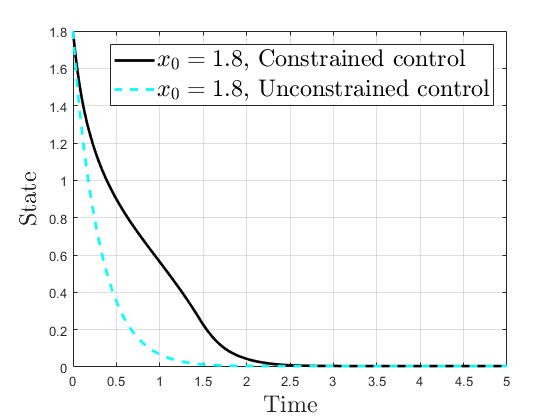}
	(iv)	\includegraphics[width=0.45\textwidth]{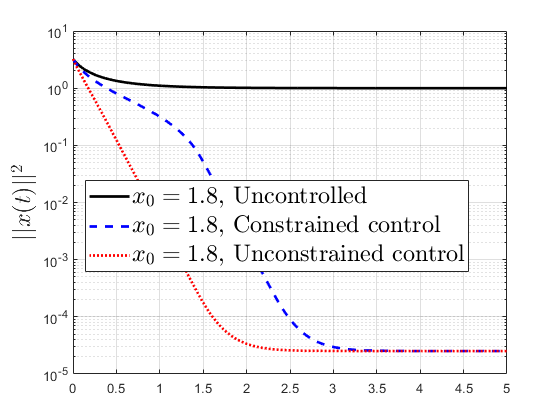}
	\caption{Test 2: Deterministic case.  {\bf i)} Value function,  {\bf ii)} Control, {\bf iii)} State, {\bf iv)} State cost. }\label{fig:ntest2}
\end{figure}
Behaviors of the value function, state, control and state cost both for the constrained and unconstrained control systems are documented in Figures \ref{fig:ntest2}(i)-(iv) with a observation of faster decay rate for trajectories in the unconstrained  cases.
\subsection{Test 3: Three dimensional linear system}
Consider the  minimization problem
\begin{align}\label{3dcost}
\underset{u(\cdot)}{\min}\; \cJ(u(\cdot),(x_0,y_0,z_0))=\int_{0}^{\infty}e^{-\lambda t}\Big(\norm{x}^2+\norm{y}^2+\norm{z}^2+\norm{u_1}^2_{R}+\norm{u_2}^2_{R}\Big) dt
\end{align}
subject to the following 3D linear control system
\begin{align}
\frac{dx}{dt}&=\sigma(y-x)\label{eqx3.4}; \quad x(0)=x_0,\\
\frac{dy}{dt}&=x\rho-y+u_1\label{eqx3.5};\quad -1\leq u_1\leq 1,\quad y(0)=y_0,\\
\frac{dz}{dt}&=-\beta z+u_2\label{eqx3.6};\quad -1\leq u_2\leq 1, \quad z(0)=z_0,
\end{align}
with $\sigma=10$, $\rho=1.1$, and $\beta=8/3$. System  \eqref{eqx3.4}-\eqref{eqx3.6} arises from
 linearization  of the Lorenz system at (0,0,0). For $\rho>1$, the equilibrium solution $(0,0,0)$ is unstable \cite[Chapter 1]{S82}.
\begin{figure}[!h]
	\centering
	(i)\includegraphics[width=0.45\textwidth]{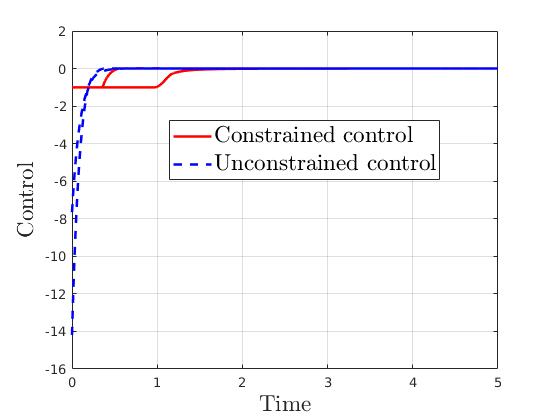}
	(ii)\includegraphics[width=0.45\textwidth]{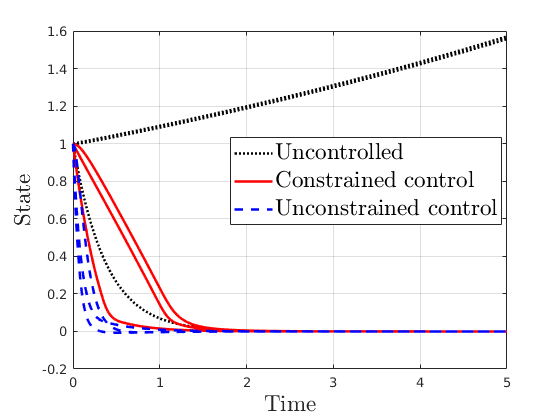}
(iii)\includegraphics[width=0.45\textwidth]{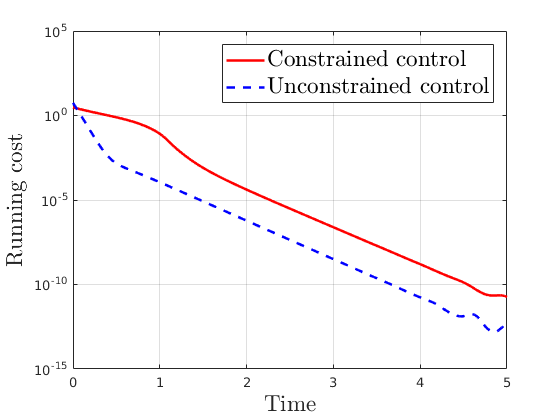}
	(iv)\includegraphics[width=0.45\textwidth]{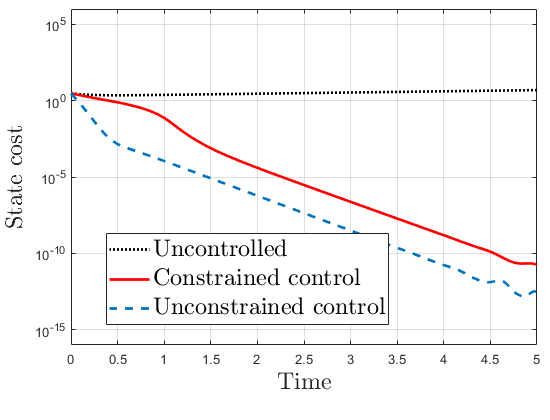}
%(v)	\includegraphics[width=0.45\textwidth]{controlcost3dlnd}
%(v)	\includegraphics[width=0.45\textwidth]{controlcostlog}
	\caption{Test 3: Deterministic case, 3D Linear, $R=0.01$,Time step $dt=10$, Initial condition $[x_0,y_0,z_0]=[1,1,1]$.  {\bf i)} Control, {\bf ii)} State, {\bf iii)} Running cost $(\norm{x(t)}^2+\norm{y(t)}^2+\norm{z(t)}^2+\norm{u_1(t)}^2_{R})+\norm{u_2(t)}^2_{R})$, {\bf iv)} State cost $\norm{x(t)}^2+\norm{y(t)}^2+\norm{z(t)}^2$ . }\label{fig:test2.1}
\end{figure}
%\todo{check figure labels}

Figure \ref{fig:test2.1} (i) depicts the unconstrained and the constrained  controls, where one component of the control constraint is active up to $t=0.5$, and the other one up to $t=1$. The first two  components of the state of the uncontrolled solution diverge, while the third one converges, see the dotted curves in \ref{fig:test2.1}(ii).  For the unconstrained HJB control the state tends to zero with a faster decay rate than for the constrained one. See again Figure \ref{fig:test2.1} (ii). Figures \ref{fig:test2.1}(iii)-(iv) document the running  and state costs, respectively.
\\
We next turn to the stochastic version of problem  \eqref{3dcost}-\eqref{eqx3.6} with noise of intensity .05 added to the second and third components.

\begin{figure}[!h]
	\centering
%	(i)\includegraphics[width=0.45\textwidth]{Control3dlns}
	(i)\includegraphics[width=0.45\textwidth]{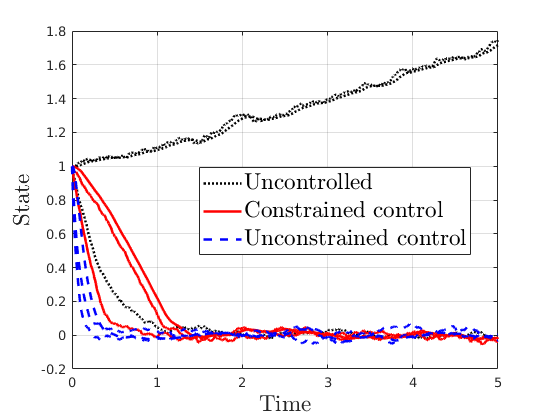}
	(ii)\includegraphics[width=0.45\textwidth]{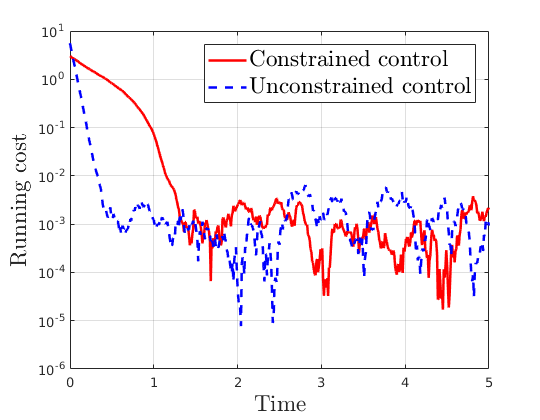}
%	(iv)\includegraphics[width=0.45\textwidth]{statecost3dlns}
%	(v)	\includegraphics[width=0.45\textwidth]{Controlcost3dlns}
	\caption{Test 3: Stochastic case, 3D Linear system, $R=0.01$,  Time step $dt=10$, Initial condition $[x_0,y_0,z_0]=[1,1,1]$.   {\bf i)} State, {\bf ii)} Running cost $(\norm{x(t)}^2+\norm{y(t)}^2+\norm{z(t)}^2+\norm{u_1(t)}^2_{R}+\norm{u_2(t)}^2_{R})$. }\label{fig:test2.2}
\end{figure}
For the stochastic system the constrained and unconstrained states tend to zero with some oscillations,
see Figure \ref{fig:test2.2}

\subsection*{Test 4: Lorenz system.}
For the third test, we choose the Lorenz system which appears in weather prediction, for example. This is a nonlinear three dimensional system, with  control appearing in  the second equation,
\begin{align}
\frac{dx}{dt}&=\sigma(y-x)\label{eq3.4}; \quad x(0)=x_0,\\
\frac{dy}{dt}&=x(\rho-z)-y+u\label{eq3.5};\quad -1\leq u\leq 1,\quad y(0)=y_0,\\
\frac{dz}{dt}&=xy-\beta z\label{eq3.6}; \quad z(0)=z_0,
\end{align}
where the three parameters $\sigma>1$, $\rho>0$ and $\beta>0$ have physical interpretation. For more details see e.g. \cite{S82}.
%Consider the minimization problem \eqref{rcost}
%\begin{align}\label{rcost}
%\underset{u(\cdot)}{\min}\cJ(u(\cdot),(x_0,y_0,z_0))=\int_{0}^{\infty}e^{-\lambda t}\Big(\norm{x}^2+\norm{y}^2+\norm{z}^2+\norm{u}^2_{R}\Big) dt
%\end{align}
%%
% subject to \eqref{eq3.4}-\eqref{eq3.6}.
 When $u=0$ in \eqref{eq3.5}, we obtain the original uncontrolled Lorenz system. The Lorenz system has 3 steady state solution namely $C0=(0,0,0)$, $C^+=\Big(\sqrt{\beta(\rho-1)}, \sqrt{\beta(\rho-1)},\rho-1\Big)$ and $C^-=\Big(-\sqrt{\beta(\rho-1)}, -\sqrt{\beta(\rho-1)},\rho-1\Big)$. For $\rho<1$, all steady state solutions are stable. For $\rho>1$, $C0$ is always unstable and $C^\pm$ are stable only for $\sigma>\beta+1$ and $1<\rho<\rho^*=\frac{\sigma(\sigma+\beta+3)}{(\sigma-\beta-1)}$. At $\rho=\rho^*$, $C^\pm$ becomes unstable. Here, we take $\sigma=10$, $\beta=8/3$ and $\rho=2$ so that $C0=(0,0,0)$ is an unstable equilibrium. We solve the HJB equation over $\Omega=(-2,2)^3$ with  $dt=0.1$,  and $R=0.01$.
\begin{figure}[!h]
	\centering
	(i)\includegraphics[width=0.45\textwidth]{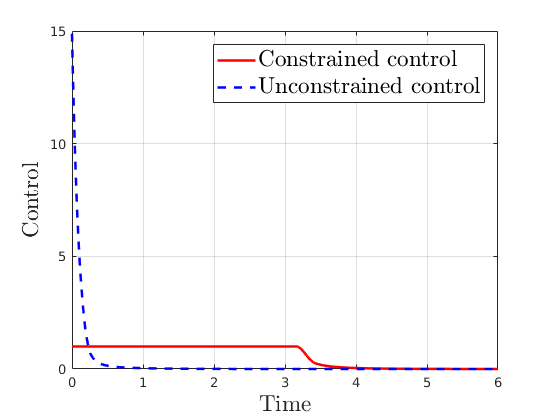}
	(ii)\includegraphics[width=0.45\textwidth]{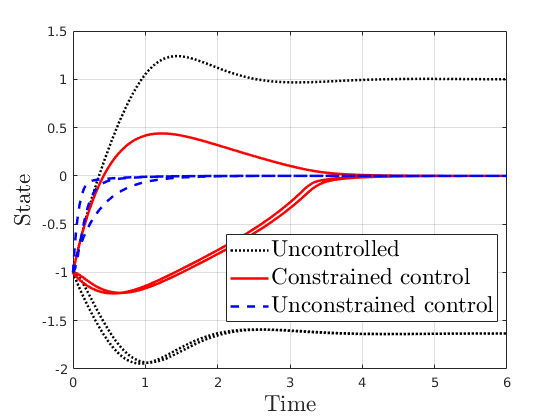}
%(iii)\includegraphics[width=0.45\textwidth]{Runningcostlorenzd}
%	(iv)\includegraphics[width=0.45\textwidth]{Statecostlorenzd}
%(v)	\includegraphics[width=0.45\textwidth]{Controlcostlorenzd}
%(v)	\includegraphics[width=0.45\textwidth]{controlcostlog}
	\caption{Test 4: Deterministic case, Lorenz system, $R=0.01$, Initial condition $[x_0,y_0,z_0]=[-1,-1,-1]$.  {\bf i)} Control, {\bf ii)} State.
%		, {\bf iii)} Running cost $\norm{y(t)}^2+\norm{u(t)}^2_{R}$, {\bf iv)} State cost $\norm{y(t)}^2$ {\bf v)} Control cost $\norm{u(t)}^2_{R}$.
	}\label{fig:test3.1}
\end{figure}
From  Figure \ref{fig:test3.1}(i), it can be observed that the unconstrained HJB control quickly tends to zero whereas the  constrained control is active up to $t\sim3$ and then converges to zero. Figure \ref{fig:test3.1}(ii) shows that
in absence of control, the state does not tend to the origin but rather to a stable equilibrium. With HJB-constrained or unconstrained control synthesis, the controlled state tends to the origin, with a faster decay rate for the unconstrained control compared to the constrained  one.\\

Similar numerical results were also obtained in the stochastic case.
\section*{Concluding remarks}
We investigated  convergence of the policy iteration technique for  the  stationary HJB equations which arise from the deterministic and stochastic control of dynamical systems in the presence of control constraints. The control constraints are realized as hard constraints by a projection operator rather than by approximation by means of
a penalty technique, for example. Numerical examples
illustrate the feasibility of the approach and provide a comparison of the behavior of the closed loop controls in the presence of control constraints and without them. The algorithmic realization is based on an upwind   scheme.
%%%%%%%%%%%%%%%%%%%%%%%%%%%%%%%%%%%%%%%%%%%%%%%%%%%%%
% References
%%%%%%%%%%%%%%%%%%%%%%%%%%%%%%%%%%%%%%%%%%%%%%%%%%%%

\section*{Acknowledgments}
The authors gratefully acknowledge support by the ERC advanced grant 668998
(OCLOC) under the EU's H2020 research program.
\end{document}